\documentclass{amsart}
\usepackage{amssymb,amsmath,amsthm,xypic,enumerate,color}

\title[Extending obstructions to noncommutative functorial spectra]{Extending obstructions to\\ noncommutative functorial spectra}
\author{Benno van den Berg}
\address{ILLC, Universiteit van Amsterdam, P.O. Box 94242, 1090 GE Amsterdam, the Netherlands}
\email{bennovdberg@gmail.com}
\author{Chris Heunen}
\thanks{Benno van den Berg was supported by the Netherlands
  Organisation for Scientific Research (NWO). Chris Heunen was
  supported by the Office of Naval Research under grant N000141010357 and by the Engineering and Physical Sciences Research Council
Fellowship EP/L002388/1.}
\address{Department of Computer Science, University of Oxford, Wolfson
  Building, Parks Road, OX1 3QD, Oxford, UK}
\email{heunen@cs.ox.ac.uk}
\date{July 9, 2014}

\newcommand{\cC}{\ensuremath{\mathcal{C}}}
\newcommand{\Cat}[1]{\ensuremath{\mathbf{#1}}}
\newcommand{\Mn}[2][n]{\ensuremath{\mathbb{M}_{#1}(#2)}}
\newcommand{\MnC}[1][n]{\Mn[#1]{\mathbb{C}}}

\newcommand{\Loc}{\Cat{Loc}}
\newcommand{\Top}{\Cat{Top}}
\newcommand{\Topos}{\Cat{Topos}}

\newcommand{\commeas}{\ensuremath{\odot}}
\newcommand{\norm}[1]{\left|\!\left| #1 \right|\!\right|}
\newcommand{\field}[1]{\ensuremath{\mathbb{#1}}}

\theoremstyle{plain}
\newtheorem{theorem}{Theorem}[section]
\newtheorem{corollary}[theorem]{Corollary}
\newtheorem{lemma}[theorem]{Lemma}
\newtheorem{proposition}[theorem]{Proposition}
\theoremstyle{definition}
\newtheorem{definition}[theorem]{Definition}
\newtheorem{example}[theorem]{Example}
\newtheorem{remark}[theorem]{Remark}

\newcommand{\cat}[1]{\ensuremath{\mathbf{#1}}}
\newcommand{\op}{\ensuremath{^{\mathrm{op}}}}
\newcommand{\C}{\ensuremath{\mathbb{C}}}
\DeclareMathOperator{\colim}{colim}
\DeclareMathOperator{\card}{card}
\DeclareMathOperator{\Proj}{Proj}

\DeclareMathOperator{\Spec}{Spec}
\DeclareMathOperator{\Sh}{Sh}

\newdir{|>}{!/2pt/@{|}*@^{>}*\dir_{>}}
\newdir{ >}{{}*!/-5pt/\dir{>}}

\begin{document}
\maketitle

\begin{abstract}
  Any functor from the category of C*-algebras to the category of
  locales that assigns to each commutative C*-algebra its Gelfand
  spectrum must be trivial on algebras of $n$-by-$n$ matrices for $n
  \geq 3$. This obstruction also applies to other spectra such as
  those named after Zariski, Stone, and Pierce. We extend
  these no-go results to functors with values in (ringed) topological spaces,
  (ringed) toposes, schemes, and quantales.
  The possibility of spectra in other categories is discussed.
\end{abstract}

\section{Introduction}

The spectrum of a commutative ring is a leading tool of commutative
algebra and algebraic geometry. For example, a commutative ring can be
reconstructed using (among other ingredients) its Zariski spectrum, a coherent topological
space. Spectra are also of
central importance to functional analysis and operator algebra. For
example, there is a dual equivalence between the category of
commutative C*-algebras and compact Hausdorff topological spaces, due
to Gelfand.\footnote{For convenience, we take all rings and
  C*-algebras to be unital, although that is not essential.}

A natural question is whether such spectra can be extended to the
noncommutative setting.
Indeed, many candidates have been proposed for noncommutative spectra.
In a recent article~\cite{reyes:onextensions}, M.\ L.\ Reyes observed
that none of the proposed spectra behave functorially, and proved that
indeed they cannot, on pain of trivializing on the prototypical
noncommutative rings $M_n(\C)$ of $n$-by-$n$ matrices with complex entries.
To be precise: any functor $F \colon \Cat{Ring}\op \to \Cat{Set}$ that
satisfies $F(C)=\Spec(C)$ for commutative rings $C$, must also satisfy
$F(M_n(\C))=\emptyset$ for $n\geq 3$.~\footnote{The rings $\MnC$ and $\C$ are Morita-equivalent, and so behave similarly in many ways. But for our purposes they are very different: for $n=3$ the theorems hold, for $n=2$ they do not.}
This result shows in a strong way why the traditional notion
of topological space is inadequate to host a good notion of
noncommutative spectrum. Its somewhat elaborate proof
is based on the Kochen--Specker Theorem \cite{kochenspecker:hiddenvariables}. It is remarkable that a
theorem from mathematical physics would have
something to say about all possible rings.

One could hope that less orthodox notions of space are less
susceptible to this obstruction. In particular, there are notions of
space, such as that of a locale or a topos, in which the notion of
point plays a subordinate role. In fact, one of the messages of locale
theory and topos theory is that one can have spaces with a rich
topological structure, but without any points whatsoever. Indeed, many
of the proposed candidate spectra for noncommutative C*-algebras have
been, or could be, phrased in such terms.

The main result of this article is that the obstruction cannot be
circumvented in this way. We will rule out many candidates for
categories of noncommutative Gelfand spectra by deriving various no-go
theorems for locales, toposes, ringed toposes, and even quantales.
Additionally, we prove similar limitative results for Zariski, Stone
and Pierce spectra.
These results will all follow from two basic ingredients.
The first is the Kochen--Specker Theorem, as in~\cite{reyes:onextensions}.
The second is a general extension theorem, prompted by our work
in~\cite{vdbergheunen:colim}, that allows us both to significantly simplify and extend Reyes' argument.

The basic obstruction is given by the Kochen--Specker Theorem. It
relates Boolean algebras to a certain noncommutative notion of Boolean
algebra. More precisely, it can be rephrased to say that any morphism
of so-called partial Boolean algebras, from the projections in $\MnC$ to a Boolean
algebra, must trivialize when $n \geq 3$.

The general extension theorem, as its name suggests, uses some simple category theory to extend this basic obstruction to far more general situations. To see how it works, consider the following commuting diagram of functors and categories.
\[\xymatrix{
  \cat{C} \ar^-{S}[r] \ar@{^{(}->}[d] & \cat{S} \ar^-{?}[d] \\
  \cat{R} \ar_-{?}[r] & \cat{?}
}\]
Here, $\cat{R}$ consists of a kind of rings, $\cat{C}$ is the full
subcategory of commutative ones, the functor $S$ takes the spectrum,
and $\cat{S}$ consists of the spectral spaces. The goal is to extend
$S$ to the noncommutative setting.
The extension theorem will state that, as long as the functor on the
right-hand side preserves limits, 
the bottom functor must degenerate.
Regarded this way, one could say that what the Kochen--Specker Theorem
obstructs, is transporting $S$ along functors whose images
have the same limit behaviour.

The paper is structured as follows. First, Section~\ref{sec:locales} motivates why it is a priori
not unreasonable to look to pointless topology for noncommutative spectra.
Section~\ref{sec:obstruction} recalls the Kochen--Specker Theorem and several
variations. Section~\ref{sec:abstract} then sets the stage with the
abstract results. After that, Sections~\ref{sec:gelfand}--\ref{sec:stone} draw corollaries of
interest from these main theorems. This host of impossibility results does
not mean that it is hopeless to search for a good notion of
noncommutative spectrum. We end the paper on a positive note by
discussing ways of circumventing the obstruction in
Section~\ref{sec:circumventing}, that will hopefully serve as a guide
towards the `right' generalization of the notion of space.

\section{Pointfree topology}\label{sec:locales}

The idea of a form of topology in which the notion of an open (or a region in space) is primary and the notion of a point plays a subordinate role dates back at least to Whitehead \cite{whitehead:anenquiry, whitehead:theconcept}. For a long time these ideas remained quite philosophical in nature and belonged to the periphery of mathematics. But this changed with the work of Grothendieck \cite{SGAIV}. The notion of a \emph{topos}, which he seems to have regarded as his most profound idea, is really a pointfree concept of a space. By now it is clear that a mathematically viable theory of pointfree spaces is possible and with topos theory this has reached a considerable degree of maturity and sophistication~\cite{johnstone:elephant, maclanemoerdijk92}.

Within the category of toposes the \emph{localic toposes} play an important role. Here they will be important because toposes that arise as spectra are localic. We will define these toposes in Definition~\ref{sheavesonalocale} below; we will have no need to consider toposes that are not localic. To define these localic toposes, the crucial observation is that in the construction of the category of sheaves over a topological spaces, the points of the space play no role. Indeed, all that matters is the structure of the lattice of opens of the space. So, to define a category of sheaves, one only needs a suitable lattice-theoretic structure. The precise structure required is formalized by the concept of a \emph{locale}, which is an important notion of pointfree space in its own right~\cite{johnstone:stonespaces,johnstone:point}.

\begin{definition} \label{locale}
  A \emph{complete lattice} is a partially ordered set of which arbitrary subsets have a least upper
  bound. In a complete lattice every subset also has a greatest lower bound. A \emph{locale} is a complete lattice that
  satisfies the following infinitary
  distributive law:
  \[
  \bigvee (x \wedge y_i) = x \wedge \bigvee y_i.
  \]
  The elements of a locale are called \emph{opens}.
  A morphism $K \to L$ of locales is a function $f \colon L \to K$ that
  satisfies $f(x \wedge y) = f(x) \wedge f(y)$ and $f(\bigvee x_i) =
  \bigvee f(x_i)$. (Note the change in direction!) This forms a category $\Cat{Loc}$.
\end{definition}

The primary example of a locale is the collection of open subsets of a
topological space. Moreover, a continuous function between topological
spaces induces a morphism between the corresponding locales (in the
same direction). Thus we have a functor $\Cat{Top} \to \Cat{Loc}$.

As it happens, this functor has a right adjoint. To construct it, define a \emph{point} of a locale $L$ as a morphism $p \colon 1 \to L$. Here, $1$ is the terminal object in the category of locales, which coincides with the set of open sets of a singleton topological space. The set of
points of $L$ may be topologized in a natural way, by declaring its
open sets to be those of the form $\{ p \mid p^{-1}(U)=1 \}$ for opens
$U$ in $L$. This defines the right adjoint $\Cat{Loc} \to \Cat{Top}$.

As usual, this adjunction becomes an equivalence if we restrict to the full subcategories of those locales and spaces for which the unit and counit are isomorphisms. These are called the \emph{spatial locales} and \emph{sober spaces}, respectively. For topological spaces, sobriety is really a weak separation property (for example, any Hausdorff topological space is sober). Thus, locales and topological spaces are closely related.

There are, however, a few subtle differences. One of the most important is that in the category of locales, limits are computed differently than in the category of topological spaces. This is one of the reasons why one might suspect that a good pointfree notion of spectrum may be possible. In fact, the following considerations may lead one to hope that a suitable notion of a pointfree space could avoid the obstruction observed by Reyes:
\begin{enumerate}
\item Many notions of spectrum lend themselves quite naturally to a pointfree formulation~\cite{johnstone:stonespaces}.
\item In many cases points correspond to maximal ideals. It is well-known that these behave very poorly functorially.
\item Limits play an important role in Reyes' result, and here as well. But limits are computed differently in topological spaces and locales. (In fact, this aspect of locales is emphasized in \cite{johnstone:point}.)
\end{enumerate}
But, as we will see below, the obstruction to nonfunctorial spectra is so fundamental that it precludes suitable notions of spectra in locales and toposes as well.

The problem is with point (3). Although limits in $\Loc$ and $\Top$ are computed differently \emph{in general}, this is not what happens with limits of locales and topological spaces that arise as spectra. There, the constructions move perfectly in tandem. This follows from the fact that locales that arise as spectra are (i) closed under limits and (ii) spatial. In fact, (i) alone already precludes the existence of suitable spectra in the category of locales, as Section~\ref{sec:abstract} below will make clear. For this reason, we will only explain (i) in some detail here.

\begin{remark} \label{constructiveness}
  Proving that locales that arise as spectra are spatial relies on nonconstructive principles, such as the Prime Ideal Theorem (a consequence of the axiom of choice). In fact, the arguments in this paper are mostly constructive: only the proofs in Section~\ref{sec:zariski} rely on results that might not be valid constructively. (That the locale-theoretic analogues of nonconstructive results in topology often are constructively valid is another aspect of locale theory emphasized in~\cite{johnstone:point}.)
\end{remark}

For example, Gelfand duality concerns compact Hausdorff spaces. Being Hausdorff is something which is rather hard to express in localic terms: but, fortunately, for compact spaces being Hausdorff is equivalent to being regular, and regularity is more readily expressed in localic terms~\cite[page~80]{johnstone:stonespaces}.

\begin{definition}\label{compactandregular}
  A locale $L$ is called \emph{compact} if any subset $S \subseteq L$ whose least upper bound is the top element has a finite subset whose least upper bound is also the top element.

  If $a$ and $b$ are two elements of a locale $L$, then $a$ is \emph{well inside} $b$ if $c \wedge a = 0$ and $c \vee b = 1$ for some $c \in L$. A locale $L$ is called \emph{regular} if any $a \in L$ is the least upper bound of the elements well inside it.
\end{definition}

\begin{lemma}\label{lem:krlocclosedunderlimits}
  Compact regular locales are closed under limits in $\Loc$.
\end{lemma}
\begin{proof}
 This follows from the fact that the the inclusion of the full subcategory $\Cat{KRLoc}$ of compact regular locales inside the category of locales has a left adjoint (namely the Stone-\v{C}ech compactification, see \cite[page~130 and page~88]{johnstone:stonespaces}).
\end{proof}

Stone duality is a duality between Boolean algebras and Stone spaces. To define the localic version of Stone spaces, observe that if $D$ is a distributive lattice, then the collection $\mathrm{Idl}(D)$ of ideals on $D$ (ordered by inclusion) is a locale. In fact, this construction is part of a functor
\[ \mathrm{Idl} \colon \cat{DLat}\op \to \Loc \]
sending ideals to the down closure of their direct images along maps of distributive lattices. This functor is faithful, but not full.

\begin{definition} \label{coherentlocale}
  A \emph{coherent locale} is one equivalent to one of the form $\mathrm{Idl}(D)$. Any coherent locale is compact; if it is also regular, we call it a \emph{Stone locale}. A map between coherent locales that is isomorphic to one in the image of the functor $\mathrm{Idl}$ is called \emph{coherent}.
\end{definition}

\begin{lemma}\label{lem:coherentlimits}
  If a diagram in $\Loc$ consists of coherent locales and coherent morphisms between them, then its limit is again a coherent locale.
\end{lemma}
\begin{proof}
  This follows from the fact that $\mathrm{Idl} \colon \cat{DLat}\op \to \Loc$ is faithful and right adjoint to the forgetful functor~\cite[page~59]{johnstone:stonespaces}.
\end{proof}

\begin{lemma} \label{lem:stonelimits}
  Stone locales are closed under limits in $\Loc$.
\end{lemma}
\begin{proof}
  This follows from Lemmas~\ref{lem:krlocclosedunderlimits} and~\ref{lem:coherentlimits}, together with the fact that every map between Stone locales is coherent~\cite[page~71]{johnstone:stonespaces}.
\end{proof}

As mentioned before, these results will preclude the existence of functorial spectra in the category of locales. They will also preclude the existence of functorial spectra in the category of toposes. Before we can explain that, let us first indicate how one can define a category of sheaves on a locale.
\begin{definition} \label{sheavesonalocale}
  A \emph{presheaf} on a locale $L$ is a functor $X \colon L\op \to \Cat{Sets}$. More concretely, a presheaf consists of a family of sets $(X(p))_{p \in L}$ together with for any $q \leq p$ a restriction operation \[ (-) \upharpoonright q: X(p) \to X(q) \] satisfying some natural compatibility conditions.

  A presheaf $X$ is a \emph{sheaf} when for any family of elements $\{ p_i \in L \mid i \in I \}$ and $\{ x_i \in X(p_i) \mid i \in I \}$ with $x_i \upharpoonright p_i \wedge p_j = x_j \upharpoonright p_i \wedge p_j$ for all $i, j \in I$ there is a unique element $x \in X(\bigvee p_i)$ with $x \upharpoonright p_i = x_i$ for every $i \in I$.

  For any locale $L$ the sheaves on $L$, with natural transformations between them, form a topos $\Sh(L)$. A topos which is equivalent to one of this form is called \emph{localic}.
\end{definition}

The construction of taking sheaves on a locale is functorial. The crucial result that will preclude noncommutative spectra valued in toposes is the following.

\begin{lemma} \label{lem:Shpreserveslimits}
  There is a full and faithful functor $\Sh \colon \Loc \to \Topos$ that assigns to every locale the category of sheaves over that locale. It preserves limits.
\end{lemma}
\begin{proof}
  For the first statement, see~\cite[Proposition~IX.5.2]{maclanemoerdijk92}. For the second, \cite[C.1.4.8]{johnstone:elephant}.
\end{proof}

\section{The Kochen--Specker Theorem}\label{sec:obstruction}

The Kochen--Specker Theorem is a famous and important result from
the foundations of quantum mechanics. Its original intention was to
preclude the possibility of hidden variable theories, but there are
interpretational debates about whether this conclusion is valid. Its
mathematical content is important to us as an example of an obstruction, as
will be defined in the next section. It was originally stated in terms
of partial algebras, which also form a convenient starting point for us.

The idea behind partial algebras is to break an algebra into parts; each part itself is a (sub)algebra with particularly nice properties, but the cohesion between the parts is lost. This lets us, for example, think about a (noncommutative) ring in terms of its commutative parts. In general, of course, the partial algebra contains less information, precisely because the whole algebra does have cohesion between the parts. The Kochen--Specker theorem, and our results based on it, concern partial algebras; they do not analyse how much ``more cohesive'' an algebra is than the sum of its parts.

  A \emph{partial Boolean algebra} consists of a set $B$ with:
  \begin{itemize}
    \item a reflexive and symmetric binary
      (\emph{commeasurability}) relation $\commeas \subseteq B \times B$;
    \item elements $0,1 \in B$;
    \item a (total) unary operation $\lnot \colon B \to B$;
    \item (partial) binary operations $\land, \lor \colon \commeas \to B$;
  \end{itemize}
  such that every set $S \subseteq B$ of pairwise commeasurable
  elements is contained in a set $T \subseteq B$, whose elements are
  also pairwise commeasurable, and on which the above operations
  determine a Boolean algebra structure. 
  A morphism of partial Boolean algebras is a function that
  preserves commeasurability and all the algebraic structure, whenever
  defined. More precisely, we have:
  \begin{itemize}
\item $f(a) \commeas f(b)$ whenever $a \commeas b$;
\item $f(0) = 0$ and $f(1) =1$;
\item $f(a \lor b) = f(a) \lor f(b)$ and $f(a \land b) = f(a) \land f(b)$ whenever $a \commeas b$;
    \item $f(\lnot a) = \lnot f(a)$ for $a \in B$.
  \end{itemize}
%
Examples of partial Boolean algebras are ordinary Boolean algebras,
where the commeasurability relation is total (we will also call these
\emph{total Boolean algebras} for that reason), and projection
lattices of Hilbert spaces. In fact, the collection of projections
\[
  \Proj(A) = \{ p \in A \mid p^*p = p \}
\]
carries the structure of a partial Boolean algebra for every
C*-algebra $A$ (where we say that two projections are commeasurable
when they commute). The Kochen--Specker Theorem now reads as follows.

\begin{theorem}[Kochen--Specker Theorem] \label{kochenspeckerforbool}
  Let $f \colon \Proj(\MnC) \to B$ be a morphism of partial Boolean
  algebras for $n \geq 3$. If $B$
  is a (total) Boolean algebra, then it must be the terminal one
  (in which $0 = 1$).
\end{theorem}
\begin{proof}
  See~\cite{kochenspecker:hiddenvariables,redei:quantumlogic}.
\end{proof}

If $B$ is a partial Boolean algebra and we write $\cC(B)$ for the
diagram of its total subalgebras and inclusions between them, then we
can rephrase the previous theorem as follows (see
also~\cite{bell}).

\begin{corollary} \label{categorifiedkochenspecker}
  If  $n\geq 3$, then the colimit of $\cC(\Proj(\MnC))$
  in the category of Boolean algebras is the terminal Boolean
  algebra.
\end{corollary}
\begin{proof}
Suppose we have a cocone from $\cC(\Proj(\MnC))$ to $B$ in the category of Boolean algebras. Clearly, it can also be considered as a cocone in the category of partial Boolean algebras. But because the colimit of $\cC(\Proj(\MnC))$ in the category of partial Boolean algebras exists and is precisely $\Proj(\MnC)$
  (see~\cite{vdbergheunen:colim}), it follows from
  Theorem~\ref{kochenspeckerforbool} that $B$ is trivial.
\end{proof}

We will also need a variation for C*-algebras. First, we define the
appropriate partial notion.
  A \emph{partial C*-algebra} is a set $A$ with:
  \begin{itemize}
 \item a reflexive and symmetric binary
      (\emph{commeasurability}) relation $\commeas \subseteq A \times A$;
    \item elements $0,1 \in A$;
    \item (partial) binary operations $+, \cdot \colon  \commeas \to A$;
    \item a (total) involution $* \colon A \to A$;
    \item a (total) function $\cdot \colon \field{C} \times A \to A$;
    \item a (total) function $\norm{-}: A \to \field{R}$;
  \end{itemize}
  such that every set $S \subseteq A$ of pairwise commeasurable
  elements is contained in a set $T \subseteq A$, whose elements are
  also pairwise commeasurable, and on which the above operations
  determine the structure of a commutative C*-algebra. 
%
  A morphism of partial C*-algebras is a morphism $f \colon A
  \to B$ preserving the commeasurability relation and all the algebraic structure, whenever
  defined. More precisely, we have:
  \begin{itemize}
\item $f(0) = 0$ and $f(1) =1$;
\item $f(a) \commeas f(b)$, $f(a+ b) = f(a) + f(b)$ and $f(ab) = f(a) f(b)$ whenever $a \commeas b$;
    \item $f(a)^* = f(a^*)$ for $a \in A$;
    \item $f(za) = zf(a)$ for $z \in \field{C}$ and $a \in A$.
  \end{itemize}
%
Any commutative C*-algebra is an example of a partial C*-algebra, on
which the commeasurability relation is total. Moreover, for any
C*-algebra $A$, the normal elements
\[
  N(A) = \{ a \in A \mid aa^* = a^*a \}
\]
carry the structure of a partial C*-algebra (where commeasurability
means commutativity). Again, we write $\cC(A)$ for the diagram of
total subalgebras of a partial C*-algebra $A$ and inclusions between
them.


\begin{corollary} \label{kochenspeckerforcstar}
  If $n \geq 3$, then the colimit of $\cC(\MnC)$ in the category of commutative C*-algebras is the terminal C*-algebra (in which $0=1$).
\end{corollary}
\begin{proof}
Suppose we have a cocone from $\cC(\MnC)$ to $A$ in the category of commutative C*-algebras. Again, we consider this as a diagram in the category of partial C*-algebras, where the colimit of $\cC(\MnC)$ is precisely $N(\MnC)$  (see~\cite{vdbergheunen:colim}). So we obtain a map $f \colon N(\MnC) \to A$ of partial C*-algebras. By restricting $f$ to the projections we obtain a map $\Proj(f): \Proj(\MnC) \to \Proj(A)$ to which
  Theorem~\ref{kochenspeckerforbool} applies. Therefore $A$ must be the
  terminal C*-algebra.
\end{proof}

\section{Obstructions}\label{sec:abstract}

This section develops a completely general way to extend obstructions
like that of the previous section. We start with the general extension
theorem, and then formalize obstructions in suitable abstract terms.

\begin{proposition}\label{prop:abstract}
  Suppose given a commuting diagram of categories and functors
  \[\xymatrix{
    \cat{A} \ar^-{F}[r] \ar_-{H}[d] & \cat{B} \ar^-{K}[d] \\
    \cat{C} \ar_-{G}[r] & \cat{D}
  }\]
  where $\cat{B}$ is complete, and $K$ preserves limits.
  If
  \begin{itemize}
  \item $\mathcal{A}$ is a diagram in $\cat{A}$,
  \item there is a cone from $X$ to $H \mathcal{A}$ in $\cat{C}$,
  \item $Y=\lim F \mathcal{A}$,
  \end{itemize}
  then there exists a morphism $G(X) \to K(Y)$ in $\cat{D}$.
\end{proposition}
\begin{proof}
  Because $K$ preserves limits, $K(Y)=K(\lim F\mathcal{A})=\lim
  KF\mathcal{A}$. The square above commutes, therefore $K(Y)=\lim
  GH\mathcal{A}$. By assumption, there is a cone from $X$ to $H\mathcal{A}$ in $\cat{C}$. Hence, there is a cone from $GX$ to $GH\mathcal{A}$
 in $\cat{D}$. But we already saw that $K(Y)$ is the target of the
  universal such cone. Hence there exists a unique mediating morphism $G(X) \to K(Y)$.
\end{proof}

Notice that the assumptions of the previous proposition were stronger than
necessary: $\cat{B}$ need not be complete, we only really need
$\lim F\mathcal{A}$ to exist in $\cat{B}$. Here is an illustration
of the situation (that will turn out not to be obstructed).

\begin{example}
  This illustration works best with colimits instead of limits, so we
  will work in the opposite setting of the previous proposition.
  Let $\cat{A}$ be the category of finite sets and injective functions, included in the
  category $\cat{C}$ of all sets and injections. Take $\cat{D}$ to be
  the ordered class of cardinal numbers, regarded as a category, and let
  $\cat{B}$ be its subcategory of at most countable cardinals, and $K$
  the inclusion. Finally, set $F$ and $G$ to be the functors that
  take cardinality. Then $\cat{B}$ is cocomplete, and $K$ preserves colimits.

  Clearly, every set $X$ is the colimit in $\cat{C}$ of the directed
  diagram $\mathcal{A}$ in $\cat{A}$ of its finite subsets and inclusions amongst
  them. If $X$ is finite, then $Y=\colim F\mathcal{A}=\sup_{A \in \mathcal{A}} \card(A)
  =\card(X)$, giving a morphism $K(Y) \to G(X)$ in $\cat{D}$. If
  $X$ is infinite, then $Y=\sup_{A \in \mathcal{A}} \card(A)$ is at
  most countable, and therefore there still is a morphism $Y
  \leq \card(X)$ in $\cat{D}$.
  %
  %
\end{example}

We can think of the previous proposition as saying that the existence
of (universal) cones to diagrams in $\cat{A}$ can be transported along
the functors $F$ and $G$.
Next, we turn to formalizing obstructions to such extensions in the
language of the previous proposition. (We are using obstruction here in the normal colloquial sense; no analogy with algebraic topology is intended.)

\begin{definition}\label{def:obstruction}
  In the situation of Proposition~\ref{prop:abstract}, an
  \emph{obstruction} to an object $X$ in $\cat{C}$ is a diagram
  $\mathcal{A}$ in $\cat{A}$ together with a cone from $X$ to $H\mathcal{A}$ in $\cat{C}$ such that $\lim F\mathcal{A}$
  is initial in $\cat{B}$. The object $X$ is called \emph{obstructed}
  if an obstruction to it exists.
\end{definition}

As a final abstract result, we now consider what happens when we try
to extend obstructed objects using Proposition~\ref{prop:abstract}.
An initial object is \emph{strict} when any
morphism into it is an isomorphism.

\begin{theorem}\label{thm:obstruction}
  In the situation of Proposition~\ref{prop:abstract}:
  if $K$ preserves initial objects,
  and initial objects in $\cat{D}$ are strict,
  then $G$ maps obstructed objects to initial objects.
\end{theorem}
\begin{proof}
  Let $X$ be an obstructed object in $\cat{C}$. Then there are a
  diagram $\mathcal{A}$ in $\cat{A}$ and a cone from $X$ to $H\mathcal{A}$
 in $\cat{C}$ such that $Y=\lim F\mathcal{A}$ is initial.
  Proposition~\ref{prop:abstract} now provides a morphism $G(X) \to K(Y)$ in $\cat{D}$.
  But since $K$ preserves initial objects, $K(Y)$ is initial in
  $\cat{D}$, and in fact strictly so. Hence the morphism
  $G(X) \to K(Y)$ must be an isomorphism, making $G(X)$ into a
  (strict) initial object.
\end{proof}

The previous theorem provides an intuition behind Definition~\ref{def:obstruction}:
whereas $X$ supports a cone to $H \mathcal{A}$,
this cone trivialises when transported along $G$.

\section{Gelfand spectrum}
\label{sec:gelfand}

This section is the first of several deriving no-go results. It shows
that there can be no nondegenerate functor extending Gelfand duality
that takes values in locales, topological spaces, toposes, or quantales.

For us, Gelfand duality is best considered as a duality between the category $\Cat{cCstar}$ of commutative C*-algebras and the category $\Cat{KRLoc}$ of compact regular locales. This duality exhibits every commutative C*-algebra $A$ as isomorphic to one of the form $\{ f: X \to \mathbb{C} \, : \, f \mbox{ continuous} \}$ for some compact regular locale $X$; the opens of the locale $X$ can be chosen to be the closed ideals of the commutative C*-algebra $A$, ordered by inclusion.

Combining the extension of Section~\ref{sec:abstract} with the
obstruction of Section~\ref{sec:obstruction}, we now immediately find
that there can be no nondegenerate functor from C*-algebras to locales
that extends the Gelfand spectrum.

\begin{corollary}\label{cor:nogo:gelfandloc}
  Any functor $G \colon \Cat{Cstar}\op \to \Loc$ that assigns to each
  commutative C*-algebra its Gelfand spectrum trivializes on $\MnC$ for $n \geq 3$.
\end{corollary}
\begin{proof}
  We instantiate the setting of Proposition~\ref{prop:abstract} by
\[ \xymatrix{ \Cat{cCstar}\op \ar[r]^\Spec \ar@{ >->}[d] &\Cat{KRLoc} \ar@{ >->}[d]^K \\
\Cat{Cstar}\op \ar[r]_G & \cat{Loc}. }\]
 By Lemma~\ref{lem:krlocclosedunderlimits},
  $\cat{KRLoc}$ is complete and $K$ preserves limits. Considering $X = \MnC$ in $\Cat{CStar}$ and $\cC(\MnC)$ in $\Cat{cCStar}$, it follows from the fact that $\Spec$ is part of a duality, and hence preserves limits, in combination with Corollary \ref{kochenspeckerforcstar} that $X$ is obstructed when $n \geq 3$. Since the initial
  object in $\Cat{KRLoc}$ and $\Loc$ is the locale of opens of the empty topological
  space, which is a strict initial object in both categories, the statement follows from
  Theorem~\ref{thm:obstruction}.
\end{proof}

\begin{remark}
  In fact, any functor as in the previous corollary must trivialize on
  many more objects than just $\MnC$ for $n \geq 3$.  For example, one
  easily derives that any C*-algebra $A$ allowing a morphism $\MnC \to
  A$ for $n \geq 3$ is also obstructed.  These are precisely those
  C*-algebras of the form $\Mn{B}$ for $n \geq 3$ and any C*-algebra
  $B$~\cite[Corollary~17.7]{lam:modules}.  Therefore, more generally,
  direct sums $\bigoplus_i \Mn[n_i]{B_i}$ are also ruled out when $n_i
  \geq 3$ for each $i$.  Any von Neumann algebra without direct
  summands $\C$ or $\MnC[2]$ is obstructed,
  too~\cite{doering:kochenspecker}.
  This remark holds for all corollaries to follow.
\end{remark}

Because of the aforementioned equivalence between the categories of
compact Hausdorff spaces and compact regular locales, the
previous corollary holds equally well for topological spaces.

\begin{corollary}\label{cor:nogo:gelfandtop}
  Any functor $G \colon \Cat{Cstar}\op \to \Top$ that assigns to each
  commutative C*-algebra its Gelfand spectrum trivializes on $\MnC$
  for $n \geq 3$.
  \qed
\end{corollary}

Since $\MnC$ and all its sub-C*-algebras are von Neumann algebras, the previous two results also holds for von Neumann algebras:

\begin{corollary}\label{cor:nogo:gelfandlocvonneumann}
  Any functor $G \colon \Cat{Neumann}\op \to \Loc$ or $G \colon \Cat{Neumann}\op \to \Top$ that assigns to each
  commutative von Neumann algebra its Gelfand spectrum trivializes on
  $\MnC$ for $n \geq 3$.
  \qed
\end{corollary}

Because a locale is a reasonably elementary geometric notion, one
might hold out hope for nondegenerate functorial extensions valued in
categories of more involved geometric objects. However, we can use
Corollary~\ref{cor:nogo:gelfandloc} as a stepping stone to derive
no-go results for the more involved geometric notions of toposes and quantales.

\begin{corollary}\label{cor:nogo:gelfandtopos}
  Any functor $G \colon \Cat{Cstar}\op \to \Cat{Topos}$ that assigns
  to each commutative C*-algebra its Gelfand spectrum trivializes on
  $\MnC$ for $n \geq 3$.
\end{corollary}
\begin{proof}
Since both the inclusion $\Cat{KRLoc} \to \Loc$ and $\Sh \colon \Cat{Loc} \to \Cat{Topos}$ preserve
  limits (see Lemmas \ref{lem:krlocclosedunderlimits} and \ref{lem:Shpreserveslimits}, respectively), their composition does as well.
  Therefore, the proof of Corollary~\ref{cor:nogo:gelfandloc} applies
  when we put $\Cat{Topos}$ in the bottom right corner.
\end{proof}

The previous corollary might not have come as a
surprise after Corollary~\ref{cor:nogo:gelfandloc}. After all, if
locales are `not noncommutative enough' to accommodate a good notion of
noncommutative Gelfand spectrum, then why would the `equally
not noncommutative' toposes do so? We will now consider quantales, which
were intended to be noncommutative versions of locales. In fact, quite
some effort has gone into studying them as candidates for Gelfand
spectra of noncommutative C*-algebras~\cite{mulvey,krumletal:max}.
The proof of the previous corollary shows that there is no
nondegenerate extension of the Gelfand spectrum with values in any
category of which compact regular locales are a
subcategory that is closed under limits. We can use the same idea in
the following.

A \emph{quantale} is a partially ordered set $Q$ that has least upper
bounds of arbitrary subsets, and is equipped with an element $e \in Q$
and an associative multiplication $Q \times Q \to Q$ satisfying
the following equations:
\[
  \bigvee (x y_i) = x (\bigvee y_i),
  \qquad
  \bigvee (y_i x) = (\bigvee y_i) x,
  \qquad
  ex = x = xe.
\]
A morphism $Q \to Q'$ of quantales is a function $f \colon Q' \to Q$
satisfying $f(e)=e'$, $f(\bigvee x_i) = \bigvee f(x_i)$, and $f(xy) = f(x)f(y)$.
Any locale is a quantale when we take meet as multiplication and the
top element as unit. Hence we can regard the Gelfand spectrum as a
functor $\Cat{cCstar}\op \to \Cat{Quantale}\op$.

\begin{lemma}\label{lem:krlocclosedunderlimitsinquantale}
  Compact regular locales are closed under limits
    in $\Cat{Quantale}\op$.
\end{lemma}
\begin{proof}
  See~\cite[Corollary~4.4]{krumletal:max}.
\end{proof}

\begin{corollary}
  Any functor $G \colon \Cat{Cstar}\op \to \Cat{Quantale}\op$ that
  assigns to each commutative C*-algebra its Gelfand spectrum
  trivializes on $\MnC$ for $n \geq 3$.
\end{corollary}
\begin{proof}
  Using Lemma~\ref{lem:krlocclosedunderlimitsinquantale} instead of
  Lemma~\ref{lem:krlocclosedunderlimits}, the proof of
  Corollary~\ref{cor:nogo:gelfandloc} establishes the statement.
\end{proof}

At first sight the previous corollary might seem to contradict results
of~\cite{krumletal:max}: one can reconstruct the original C*-algebra
from its maximal spectrum, and the assignment which sends a C*-algebra
to its maximal spectrum is functorial. However, this functor does not
send a commutative C*-algebra to its Gelfand spectrum, but
to something from which it may be reconstructed (its so-called
spatialization). Therefore the maximal spectrum does not satisfy our
specification square of Proposition~\ref{prop:abstract}.

\section{Zariski spectrum}
\label{sec:zariski}

In this section we turn to the Zariski spectrum. This construction underlies
algebraic geometry by connecting commutative rings to coherent spaces via
the prime ideals of the ring~\cite{eisenbud,johnstone:stonespaces}; more precisely, the Zariski spectrum of a commutative ring $A$ is the locale whose opens are the radical ideals of $A$. Before we go on to extending obstructions to noncommutative
generalizations of this duality, we first consider the basic no-go
result. The abstract machinery from Sections~\ref{sec:obstruction}
and~\ref{sec:abstract} does not apply directly, because the Zariski
spectrum functor $\Cat{cRing}\op \to \Loc$ famously does not preserve
(products and hence) limits. Fortunately, it suffices to restrict to
finite-dimensional complex algebras, where the Zariski spectrum
functor does preserve limits, and where our obstructed objects $\MnC$
for $n \geq 3$ live.

\begin{corollary}\label{cor:nogo:zariskiloc}
  Any functor $G \colon \Cat{Ring}\op \to \Loc$ that assigns to
  each commutative ring its Zariski spectrum trivializes on $\MnC$ for
  $n \geq 3$.
\end{corollary}
\begin{proof}
  When a commutative algebra $A$ over $\mathbb{C}$ is
  finite-dimensional, it is Artinian as a ring, and therefore any
  prime ideal is maximal~\cite[Theorem~2.14]{eisenbud}. In particular,
  every point in $\Spec(A)$ is closed. In turn, maximal ideals correspond
  bijectively, and functorially, to algebra homomorphisms: a character
  $f \colon A \to \mathbb{C}$ corresponds to its kernel
  $f^{-1}(0)$. Thus, when restricted to finite-dimensional commutative
  complex algebras, the Zariski spectrum functor is naturally
  isomorphic to a representable functor: $\Spec \cong
   \cat{cRing}(-,\mathbb{C}) \colon \cat{fcAlg}_{\mathbb{C}}\op \to
  \Cat{Set}$. Moreover, in this case there are only
  finitely many maximal ideals~\cite[Theorem~2.14]{eisenbud}, so
  $\Spec(A)$ must be discrete. Clearly
  discrete locales are closed under limits in $\Loc$ (see also Lemma~\ref{lem:coherentlimits}), so this
  restricted functor preserves finite limits, and just as in
  Corollary~\ref{cor:nogo:gelfandloc}, we see that any
  functor $\cat{fAlg}_{\mathbb{C}}\op \to \Cat{Loc}$ that assigns to each
  commutative algebra its Zariski spectrum must trivialize on $\MnC$
  for $n \geq 3$. Precomposing with the inclusion  $\cat{fAlg}_{\mathbb{C}} \hookrightarrow \cat{Ring}$ finishes the proof.
\end{proof}

Reyes' result~\cite{reyes:onextensions} now follows directly from the
previous, constructive, corollary.

This basic no-go result can be extended to values in categories of
which coherent locales are a subcategory that is
closed under limits, as in Section~\ref{sec:gelfand}. For
example, we get the following corollary.

\begin{corollary}\label{cor:nogo:zariskitopos}
  Any functor $G \colon \Cat{Ring}\op \to \Cat{Topos}$ that assigns to
  each commutative ring its Zariski spectrum trivializes on $\MnC$ for
  $n \geq 3$.
  \qed
\end{corollary}

In Section~\ref{sec:gelfand} we used closure under limits to extend
the basic no-go result. Another way is by postcomposing with functors
that reflect initial objects, as in the rest of this section. Incidentally,
these limitations also apply to functorial extensions of Gelfand
duality discussed in Section~\ref{sec:gelfand}.

Another generalized notion of space is that of a \emph{ringed
topological space} or \emph{ringed locale} \cite{hartshorne:book}. These are topological
spaces/locales together with a sheaf of commutative rings over
them, and are important in algebraic geometry. Every topological
space/locale $X$ can be regarded as a ringed space by letting the
structure sheaf be the sheaf of continuous functions on opens of $X$.
One can also consider the notion of a \emph{ringed topos}: a topos
together with a commutative ring object in it. This notion generalizes
those of ringed topological spaces and ringed locales, because the
category of sheaves over a ringed space is a ringed topos almost by
definition. The import lies in the fact that every commutative ring is
isomorphic to the ring of global sections of a sheaf of local rings.
Thus we can regard the Zariski spectrum as a functor
$\Cat{cRing}\op \to \Cat{RingedTop}$, $\Cat{cRing}\op \to
\Cat{RingedLoc}$, or $\Cat{cRing}\op \to \Cat{RingedTopos}$.

\begin{corollary}\label{cor:nogo:zariskiringedtopos}
  Any functor $G \colon \Cat{Ring}\op \to \Cat{RingedTopos}$ that
  assigns to each commutative ring its Zariski spectrum
  trivializes on $\MnC$ for $n \geq 3$. The same holds when we replace
  $\Cat{RingedTopos}$ by $\Cat{RingedTop}$ or $\Cat{RingedLoc}$.
\end{corollary}
\begin{proof}
  The forgetful functor $U \colon \Cat{RingedTopos} \to \Cat{Topos}$
  reflects initial objects. Since $UG$ is a functor satisfying the
  hypotheses of Corollary~\ref{cor:nogo:zariskitopos}, $UG(\MnC)$ is
  initial when $n \geq 3$. But that means that $G(\MnC)$ is initial.
\end{proof}

Actually, the main notion of interest in algebraic geometry is that of
a scheme (see \cite{hartshorne:book}). A locally ringed space is a ringed space where each
stalk of the structure sheaf is not just a ring but a local ring. An
\emph{affine scheme} is a locally ringed space isomorphic to the
Zariski spectrum of some commutative ring. A \emph{scheme} is a
locally ringed space admitting an open cover, such that the
restriction of the structure sheaf to each covering open is an affine
scheme.

\begin{corollary}\label{cor:nogo:zariskischeme}
  Any functor $G \colon \Cat{Ring}\op \to \Cat{Scheme}$ that assigns
  to each commutative ring its Zariski spectrum trivializes on
  $\MnC$ for $n \geq 3$.
\end{corollary}
\begin{proof}
  The forgetful functor from the category of schemes to $\Top$
  reflects initial objects, so the proof of the previous corollary applies.
\end{proof}

\section{Stone and Pierce spectra}\label{sec:stone}

In this section we will have a further look at some dualities related to the Stone spectrum, where the Kochen--Specker Theorem also provides an obstruction to further extending them to suitably  noncommutative  structures.

First we consider the Stone spectrum, that provides a duality between Boolean algebras and Stone locales: given a Boolean algebra, the associated Stone locale has as opens the ultrafilters on $B$; and given a Stone locale $L$, the original Boolean algebra can be reconstructed by taking the complemented elements in $L$.

\begin{corollary}\label{cor:partialbool}
  Any functor $F \colon \Cat{PBoolean}\op \to \Loc$ that assigns to
  each Boolean algebra its Stone spectrum trivializes on
  $\mathrm{Proj}(\MnC)$ for $n \geq 3$.
  \qed
\end{corollary}
\begin{proof}
  If one considers the diagram
  \[ \xymatrix{ \Cat{Boolean}\op \ar[r] \ar[d] & \Cat{Stone} \ar[d] \\
    \Cat{PBoolean}\op \ar[r]_-F & \Cat{Loc} } \]
  and the object  $\mathrm{Proj}(\MnC)$ in \Cat{PBoolean} (together with its diagram of commutative subalgebras in $\Cat{Boolean}$), we see that they are obstructed for every $n \geq 3$. Therefore they will be sent to the initial object by $F$.
\end{proof}

Traditional quantum logic, by which we mean the approach dating back to Birkhoff and von Neumann \cite{birkhoffvonneumann:logic}, considers orthomodular lattices. A lattice $L$ is called \emph{orthocomplemented}, if it comes equipped with a map $\perp \colon L \to L$ satisfying:
\begin{itemize}
\item $a \leq b \Rightarrow b^\perp \leq a^\perp$;
\item $(a^\perp)^\perp = a$;
\item $a \land a^\perp = 0$ and $a \lor a^\perp = 1$.
\end{itemize}
We call $a^\perp$ the \emph{orthocomplement} of $a$, and say that $a$ is commeasurable with $b$ (and write $a \commeas b$), if
\[ a = (a \land b) \lor (a \land b^\perp). \]
This relation is clearly reflexive, but need not be symmetric; if it is, we will call the lattice \emph{orthomodular}.\footnote{This is equivalent to the usual statement of the orthomodular law $a \leq b \Rightarrow b = a \lor (b \land a^\perp)$ by \cite[Theorem II.3.4]{beran:orthomodular}.} With lattice homomorphisms preserving orthocomplements as morphisms, orthomodular lattices form a category $\Cat{OrthoLat}$.

The previous no-go result extends to orthomodular lattices. This is due to several facts. First of all, every Boolean algebra is an orthomodular lattice. In fact, these are precisely the orthomodular lattices in which every two elements are commeasurable~\cite[Corollary II.4.6]{beran:orthomodular}.
Furthermore, projections $\Proj(\MnC)$ in $n$-dimensional complex Hilbert space can be identified with the subspace of $\mathbb{C}^n$ they project onto, and therefore form an orthomodular lattice~\cite[Section~III.4]{beran:orthomodular}: the order comes from subspace inclusion, and $\perp$ comes from orthocomplement. Now, the relation $\commeas$ gives every orthomodular lattice the structure of a partial Boolean algebra~\cite[Theorem II.4.5]{beran:orthomodular}. Projection lattices thus obtain partial Boolean algebra structure: projections $p$ and $q$ commute if and only if the subspaces $p(\mathbb{C}^n)$ and $q(\mathbb{C}^n)$ they project onto are commeasurable in the orthomodular lattice of linear subspaces~\cite[Exercise III.18]{beran:orthomodular}. Therefore also the two different notions of total (or commeasurable) subalgebra agree.

\begin{corollary}
  Any functor $\Cat{OrthoLat}\op \to \Loc$ that assigns to each Boolean algebra its Stone
  spectrum trivializes on $\mathrm{Proj}(\MnC)$ for $n \geq 3$.
  \qed
\end{corollary}
\begin{proof}
  Proved in the same way as the previous corollary, where this time we put $\Cat{OrthoLat}\op$ in the bottom left corner.
\end{proof}

Next, we turn to the Pierce spectrum, which assigns to a commutative ring the Stone space of its Boolean algebra of idempotents.

\begin{corollary}
  Any functor $\Cat{Ring}\op \to \Loc$ that assigns to each commutative ring its Pierce
  spectrum trivializes on all $\MnC$ for $n \geq 3$.
\end{corollary}
\begin{proof}
  Let $F: \Cat{Ring}\op \to \Loc$ be as in the statement. Let $\cC(\MnC)$ be the diagram of commutative self-adjoint subalgebras of $\MnC$. As usual, we will argue that $\lim F\cC(\MnC)$ in $\Loc$ is initial. Consider the restriction $\overline{F}$ of $F$ to $\Cat{cNeumann}$, and denote $G$ for the functor that sends a commutative von Neumann algebra to its Gelfand spectrum. Since every projection is an idempotent, and the Gelfand spectrum of a commutative von Neumann algebra is given by the Stone space on its projections, there is a natural transformation $\overline{F} \Rightarrow G$. So if $\lim G\cC(\MnC)$ is the (strict) initial object in $\Loc$, the same must be true for $\lim F\cC(\MnC) = \lim \overline{F}\cC(\MnC)$.
\end{proof}

\section{Circumventing obstructions}
\label{sec:circumventing}

It might be tempting to conclude from the above impossibility results that it is hopeless to look for a good notion of spectrum for noncommutative structures. But we strongly believe that this is the wrong conclusion to draw. What our results show is merely that a category of noncommutative spectra must have different limit behaviour from the known categories of commutative spectra. One of the central messages of category theory is that objects should be regarded as determined by their behaviour rather than by any internal structure. In other words, it is not the internal structure of objects that dictates what morphisms should preserve. It is the other way around: it is the morphisms connecting an object to others that determine that object's characteristics. Ideally, of course, both viewpoints coincide. But the latter viewpoint is better precisely when it is unclear what the objects should be. Historically, noncommutative spectra have almost always been pursued by generalizing the internal structure of commutative spaces (as objects). We believe the right, and optimistic, message to distill from our results is that one should let the search for noncommutative spectra be guided by morphisms instead. Indeed, the few proposals for noncommutative spectra that escape our obstructions have non-standard morphisms between them:
\begin{itemize}
\item There is a notion of noncommutative spectrum due to Akemann, Giles and Kummer~\cite{akemann:stone,gileskummer,akemann:gelfand}. It allows one to reconstruct the original C*-algebra, but the correspondence is only functorial for certain morphisms of C*-algebras.
\item The so-called process of Bohrification gives a functor from the category of C*-algebras to localed toposes~\cite{heunenlandsmanspitters:bohrification}. It involves some loss of information, however: one can only reconstruct the partial C*-algebra structure of the original C*-algebra~\cite{vdbergheunen:colim}. Indeed the natural morphisms in this setting are partial *-homomorphisms.
\item It is possible to construct a functor from the category of C*-algebras to the category of so-called quantum frames~\cite{rosicky:quantumframes}. These structures only take into account the Jordan structure of the original C*-algebra, and this is reflected in the choice of morphisms. Indeed, there is no nondegenerate functor between the categories of quantum frames and that of quantales, so there is no contradiction with our results.
\item A recent paper by Heunen and Reyes proposes a new notion of spectrum for arbitrary AW*-algebras~\cite{heunenreyes:activelattices}. It involves an action of the unitary group on the projection lattice, and therefore the natural morphisms are quite unlike those of topological spaces.
\end{itemize}

\bibliographystyle{plain}
\bibliography{nogo}

\end{document}